\documentclass{amsart}
\usepackage{graphicx}
\usepackage{amsthm}
\usepackage{amsmath}
\usepackage{amssymb}
\usepackage{color}
\newtheorem{thm}{Theorem}[section]

\theoremstyle{definition}

\theoremstyle{remark}

\newtheorem{theorem}{Theorem}

\newtheorem{lemma}{Lemma}
\newtheorem{proposition}{Proposition}
\newtheorem{definition}{Definition}
\newtheorem{corollary}{Corollary}

\newtheorem{example}{Example}
\newtheorem{problem}[thm]{Problem}
\newcommand{\RR}{\mathbb R}

\newcommand{\Iff}{if and only if}

\numberwithin{equation}{section}

\begin{document}

\title[A note on (weak) phase and norm retrievable
Real Hilbert frames]{A note on (weak) phase and norm retrievable  Real Hilbert space frames and projections}
\author{F. Akrami}
\address{Department of sciences, University of Maragheh, Maragheh, Iran.}
\email{fateh.akrami@gmail.com}
\author{P. G. Casazza}
\address{Department of Mathematics, University of Missouri, Columbia, USA.}
\email{casazzap@missoouri.edu}
\author{A. Rahimi}
\address{Department of sciences, University of Maragheh, Maragheh, Iran.}
\email{rahimi@maragheh.ac.ir}
\author{M. A. Hasankhani Fard}
\address{Department of Mathematics, Vali-e-Asr University of Rafsanjan, Rafsanjan, Iran.}
\email{m.hasankhani@vru.ac.ir}
\author{B. Daraby}
\address{Department of sciences, University of Maragheh, Maragheh, Iran.}
\email{bdaraby@maragheh.ac.ir}
\dedicatory{}

\thanks{The second author was supported by NSF DMS 1609760}

\subjclass[2010]{42C15, 42C40.}

\keywords{Real Hilbert frames, Fusion frame, Finitely full spark, Full spark,
Phase retrieval, Norm retrieval.}

\begin{abstract}
In this manuscript, we answer a list of longstanding open problems on weak phase retrieval including: (1) A complete
classification of the vectors $\{x_i\}_{i=1}^2$ in $\RR^3$ that do weak phase retrieval; (2) We show that frames doing
weak phase retrieval in $\RR^n$ must span $\RR^n$; (3) We give an example of a set of vectors doing phase
retrieval but their orthogonal complement hyperplanes fail weak phase retrieval; (4) We give a classification
of weak phase retrievable frames - which makes clear the difference between phase retrieval and weak phase
retrieval; (5) We classify
when weak phase retrievable frames also do norm retrieval. We then introduce the notion
of weak phase retrieval by projections and develop their basic properties. We then look at phase (norm) retrieval
by projections. We end with some open problems.
 We provide numerous examples to show that our results are best possible.
\end{abstract}

\maketitle

\section{Introduction}\label{s:intro}
The concept of frames in a separable Hilbert space was originally introduced by Duffin and Schaeffer in the context of non-harmonic Fourier series \cite{DS1952}.
Frames are a more flexible tool than bases because of  the redundancy property that make them more applicable than bases.
 Phase retrieval is an old problem of recovering a signal from the absolute value of linear measurement coefficients called intensity measurements.
 Phase retrieval and norm retrieval have become very active areas of research in applied mathematics, computer science, and engineering, and more today. Phase retrieval has been defined for both vectors and subspaces (projections) in a separable Hilbert spaces.
\par
The concept of weak phase retrieval weakened of notion of phase retrieval and it has just been defined for vectors (\cite{ACNT2015} and \cite{ACGJT2016}). In this paper we define the concept of weak phase retrieval by projections and make a detailed
study of weak phase retrieval by vectors and projections. The rest of the paper is organized as follows: In Section 2, we give the basic
definitions and certain preliminary results to be used in the paper. The investigation of weak
phase retrieval by vectors is in section 3 and  in Section 4 we classify weak phase retrievable frames in $\RR^2$. In Section 5, we show that weak phase retrievable frames must span the whole space.  Weak phase retrieval by projections will be introduced in section 6.  In section 7 we classify the existence of weak
phase retrievable frames.
 Some facts about phase (norm) retrieval by projections will be discussed in section 8. Finally, in section 9 we present
 some open problems.

\section{preliminaries} \label{s:pre}
We first give the background material needed for the paper. Let $\mathbb{H}$ be a finite or infinite dimensional Real Hilbert space and B$(\mathbb{H})$ the class of all bounded linear operators defined on $\mathbb{H}$. The natural numbers and real numbers are denoted by $``\mathbb{N}"$ and $``\mathbb{R}"$, respectively. We  use $[m]$ instead of the set $\{1,2,3,\dots,m \}$ and use $[\{x_i\}_{i \in I}]$ instead of $span\{x_i\}_{i \in I}$ where $I$ is a finite or countable subset of $\mathbb{N}$. We denote by $\mathbb{R}^n$ a $n$ dimensional real  Hilbert space.
 We start with the definition of a real Hilbert space frame.
\begin{definition} \label{D:frame}
 A family of vectors $\{x_i\}_{i\in I}$ in a finite or infinite dimensional separable Hilbert space $\mathbb {H}$ is a \textbf{frame} if there are constants $0<A \leq B< \infty $ so that
$$ A\|x\|^2 \leq \sum_{i\in I}{|\langle x,x_i \rangle|^2}\leq B\|x\|^2, \quad \mbox{for all} \quad f\in \mathbb{H}.$$
The constants $A$  and $B$ are called the lower and upper frame bounds for $\{x_i\}_{i\in I}$, respectively. If an upper frame bound exists, then $\{x_i\}_{i\in I}$ is called a {\bf B-Bessel set} or simply
{\bf Bessel} when the constant is implicit. If $A=B$, it is called an {\bf A-tight frame} and
 in case $ A=B=1$, it called a {\bf Parseval~frame}.  The values $\{\langle x,x_i \rangle\}_{i=1}^{\infty}$ are called the frame coefficients of the vector $x \in \mathbb{H}$.
\end{definition}
\begin{definition}\label{D: fusion frame}
A family $\{(W_{i},v_{i}) \}_{i \in I }$ with $W_i$ subspaces of $\mathbb{H}$, $v_i$ weights, and $P_i$ the
projection onto $W_i$, is a\textbf{ fusion frame} for $\mathbb{H}$  if  there exist constants
$A,B > 0$ such that
$$A\|x\|^2 \leq \sum_{i \in I}v_{i}^{2}\|P_i x\|^2 \leq B\|x\|^2, \mbox{ for all } x\in \mathbb{H}.$$ The constants $A$ and $B$ are called the
\textbf{fusion frames bounds}.  We also refer to the fusion frames as $\{P_i,v_i\}_{i\in I}$ or just $\{P_i\}_{i\in I}$ if the weights are all one.
\end{definition}
\begin{definition}\label{D:phase(norm) retrieval by proj}
\textrm A family of projections $\{P_{i}\}_{i\in I}$ in a real Hilbert space $\mathbb {H}$ does \textbf {phase (norm) retrieval} if whenever $ x, y \in \mathbb {H}$, satisfy
$$\|P_ix\|=\|P_iy\| \quad \mbox{ for all } i\in I,$$
then $x=\pm y \;( \|x\|=\|y\|)$. \\ [10pt]
\end{definition}

\begin{definition}\label{D:phase(norm) retrievable fusion frame}
\textrm {A fusion frame} $\{(W_{i},v_{i}) \}_{i \in I }$ \, \textrm{is phase (norm) retrievable for} $\mathbb{H}$ \, \Iff \, the family of projections $\{P_{i} \}_{i\in I}$ is phase (norm) retrievable for
$\mathbb{H}$, where $P_{i}=P_{W_{i}}$ is the orthogonal projection onto $\{W_{i},\}(i \in I)$.
\end{definition}

For more details on fusion frames, we recommend \cite{CK2004}.
It is well known that phase (norm) retrievable sets need not be fusion frames.for example consider
$\{e_i + e_j\}_{i<j}$ that does not satisfy in the frame upper bound condition and therefore is not a frame(fusion frame)for $\ell^2$.
 Also the converse is not true in general. e.g., \cite{AGA2015} let
 $\{W_j\}_{j=1}^{\infty}=span\{e_{2j-1},e_{2j}\}_{j=1}^{\infty}$ be a countable collection subspaces of $\ell^2$ with corresponding orthogonal projections $P_jx= \langle x,e_{2n-1} \rangle e_{2n-1}+ \langle x,e_{2n} \rangle e_{2n}$, and positive weights $\{v_i\}_{i=1}^{\infty}=\{(\frac{1}{2})^j\}_{j=1}^{\infty}$,  then $\{(W_i,v_i)\}_{i=1}^{\infty}$ is a fusion frame for $\ell^2$, but this set is not phase (norm) retrievable.\\
 Improving and extending the notions of phase  and norm  retrievability, we present the definition of phase (norm) retrievable to fusion frames.

We will need to work with Riesz sequences.
\begin{definition} \label{D:Riesz sequence}
A family $\Phi = \{x_i\}_{i \in I}$ in a finite or infinite dimensional Hilbert space $\mathbb{H}$
is a \textbf{Riesz sequence} if there are constants $0<A \leq B< \infty $ satisfying
$$A\sum_{i \in I}|c_i|^2 \leq \|\sum_{i\in I}{c_i x_i}\|^2\leq B\sum_{i \in I}|c_i|^2$$
 for all sequences of scalars $\{c_i\}_{i \in I}$.
If it is complete in $\mathbb{H}$, we call $\Phi$ a \textbf{Riesz basis}.
\end{definition}

\begin{definition}\label{D:phase(norm) retrieval by vectors}
\textrm A family of vectors $\{x_{i}\}_{i\in I}$ in a real Hilbert space $\mathbb {H}$ does \textbf {phase (norm) retrieval} if whenever $ x, y \in \mathbb {H}$, satisfy
$$|\langle x,x_{i}\rangle |=|\langle y,x_{i}\rangle| \quad \mbox{ for all } i\in I,$$
then $x=\pm y \;( \|x\|=\|y\|)$. \\ [10pt]
\end{definition}
Note that if $\{x_i\}_{i\in I}$ does phase (norm) retrieval, then so does $\{a_ix_i\}_{i\in I}$ for any $0< a_i< \infty$ for all $i\in I$.
But in the case where $|I|=\infty$, we have to be careful to maintain frame bounds. This always works if $0<\inf_{i\in I}a_i \le sup_{i\in I}a_i < \infty$.
But this is not necessary in general \cite{ACHR2021}.

 The complement property is an essential here.

\begin{definition}\label{D:complement pro in infinite case}
\textrm A family of vectors $\{x_{k}\}_{k\in I}$ in a finite or infinite dimensional Hilbert space $\mathbb{H}$ has the \textbf{complement property}\, \textrm  if for any subset $I \subset \mathbb{N}$, \\
$$either\ \  \overline{span}\{x_{k}\}_{k\in I}=\mathbb{H} \textrm \quad or \quad \ \  \overline{span}\{x_{k}\}_{k\in I^c}=\mathbb{H}. $$
\end{definition}
Fundamental to this area is the following for which the finite dimensional case appeared in \cite{CCPW13}
and the infinite dimensional case appeared in \cite{CCD16}.
\begin{theorem}\label{T:phase retrievality and complement property for finite case}
A family of vectors $\{x_{i}\}_{i\in I} $ does phase retrieval if and only if it has the complement property.
\end{theorem}

We recall

\begin{definition}\label{D:full spark}
\textrm A family of vectors $\{x_{i}\}_{i=1}^{m}$ \textrm in $\mathbb{R}^n \, (m \geq n)$ has {\bf spark k} if for every $I \subset [m] \, \mbox with \, |I|=k-1$ \textrm, $\{x_{i}\}_{i \in I}$ \textrm
is linearly independent. It is full spark if $k=n+1$ and hence every n-element subset spans  $\mathbb{R}^n$.
\end{definition}
\begin{corollary}\label{T:phase retrievality and complement property and full apark for finite case}
\textrm If $\{x_{i}\}_{i=1}^{m} $  \textrm does phase retrieval in $\mathbb {R}^n$, then $m \geq2n-1$.
If $m\ge 2n-1$ and the frame does phase retrieval if and only if it is full spark.
\end{corollary}

For linearly independent sets there is a special case \cite{CCJW14}.
\begin{theorem}\label{TT}
If $\{x_i\}_{i=1}^n$ in $\RR^n$ does norm retrieval, then the set is orthogonal.
\end{theorem}
A similar result in $\ell^2$ in \cite{ACHR2021}.
That is, if a Riesz basis  $\{x_i\}_{i=1}^{\infty}$   does norm retrieval in $\ell^2$, then the vectors $\{x_i\}_{i=1}^{\infty}$ are orthogonal.

It is clear that phase retrieval implies norm retrieval. The converse fails since an orthonormal basis does norm retrieval but fails phase
retrieval since it fails complement property.
Subsets of phase (norm) retrievable frames certainly may fail phase (norm) retrieval, since linearly independent subsets fail the complement property so
fail phase retrieval and
by Theorem \ref{TT} if every subset of a frame does norm retrieval then every two distinct vectors are orthogonal and so the frame is an orthogonal set plus possibly
more vectors.
 But
projections of these sets do still do phase (norm) retrieval \cite{ACHR2021}.

 It is well known that every finite dimensional real Hilbert space  $\mathbb{H}$ \textrm is isomorphic to $\mathbb{R}^n$, for some n, and every separable infinite
dimensional real Hilbert space  $\mathbb{H}$ \textrm is  isomorphic to $\ell^2(\mathbb{R})$ (countable real sequences with $\ell^2$~-norm). We will use $\ell^2$ instead of $\ell^2(\mathbb{R})$ for simplicity.
Throughout the paper, $\{e_i\}_{i=1}^{\infty}$ will be used to denote the canonical basis for the real space $\ell^2$, i.e., a basis for which
$$\langle e_i,e_j \rangle=\delta_{i,j}=
\begin{cases}
1 \quad  if \ i=j,  \\
0 \quad  if \ i\ne j.
\end{cases}$$

\section{Weak phase retrieval by vectors}\label{s:Weak phase retrieval by vectors and projections }
The notion of ``Weak phase retrieval by vectors'' in $\mathbb{R}^n$ was introduced in \cite{ACNT2015} and was
developed further in \cite{ACGJT2016}. For $x\in \RR^n$, $sgn(x)=1$ if $x>0$ and $sgn(x)=-1$ if $x<0$.
\begin{definition}\label{D:Weakly have the same phase}
Two vectors  $x=(a_1,a_2,\dots,a_n)$ and $y=(b_1,b_2,\dots,b_n)$ in $\mathbb{R}^n$  
{\bf weakly have the same phase} if there is a $|\theta|=1$ so that
$phase(a_i)=\theta phase(b_i)$ for all $i \in [n]$, for which $a_i \ne 0 \ne b_i$.\\
If $\theta=1$, we say $x$ and $y$ weakly have the same signs and if $\theta=-1$, they weakly have the opposite signs.
\end{definition}
Therefore with above definition the zero vector in $\RR^n$  weakly has the same phase with all vectors in $\RR^n$.
\begin{definition}\label{D:Weak phase retrieval by vectors}
A family of vectors $\{\phi_i\}_{i=1}^m$ does {\bf weak phase retrieval} in $\mathbb{R}^n$ if for any $x=(a_1,a_2,\dots,a_n)$ and $y=(b_1,b_2,\dots,b_n)$ in $\mathbb{R}^n$ with $|\langle x,\phi_i \rangle |=|\langle y,\phi_i \rangle |$ for all $i \in [m]$, then $x$ and $y$ weakly have the same phase.
\end{definition}
A fundamental result here is
\begin{proposition}  \label{D:pp}\cite{ACNT2015}
Let $x=(a_1,a_2,\dots,a_n)$ and $y=(b_1,b_2,\dots,b_n)$ in $\mathbb{R}^n$. The following are equivalent:\\
(1) We have $sgn(a_i a_j)=sgn(b_i b_j)$,\: for  all $0 \leq i \ne j \leq n$ \\
(2) Either $x,y$ have weakly the same sign or they have the opposite signs.
\end{proposition}
It is clear that if $\{x_i\}_{i=1}^m$ does phase retrieval (respectively, weak phase retrieval) in $\RR^n$ then
$\{c_ix_i\}_{i=1}^n$ does phase retrieval (respectively, weak phase retrieval) as long as $c_i>0$ for all $i=1,2,\ldots,m$.

The following appears in \cite{ACGJT2016} 
\begin{theorem}
If $\{x_i\}_{i=1}^m$ does weak phase retrieval in $\mathbb{R}^n$ then $m \geq 2n-2$.
\end{theorem}

\section{Classifying Weak Phase Retrievable Frames in $\RR^2$}
\begin{theorem}\label{T10}
A set of vectors $\{x_i\}_{i=1}^2$ does weak phase retrieval in $\RR^2$ if and only if after normalization the two vectors
are of the form $(1,1),\ (1,-1)$ or $(1,b),\ (1,-b)$.
\end{theorem}

 We will prove the theorem in a series of 3 lemmas.

 \begin{lemma}\label{T:X5}
 Sets $\{(a,0),\ (b,c)\}$ and $\{(0,b),\ (c,d)\}$ with $a,b,c,d\not= 0$ fail weak phase retrieval but they are full spark.
   \end{lemma}

 \begin{proof}
 We look at two cases for the first choice. Note that $c\not= 0$ or after normalizing we have only one vector
 and as we saw above this cannot do weak phase retrieval.
 \vskip12pt
 \noindent {\bf Case 1}: After normalizing we have vectors $x_1=(1,0),\ x_2=(1,a)$ with $0<a$.

 Now, define $x=(a_1,a_2)=(1,1)$ and $y=(b_1,b_2)=(1,-\frac{2}{a}-1)$. Then $a_1a_2>0$ while $b_1b_2<0$.
 Also,
 \[ |\langle x,x_1\rangle|=1=|\langle y,x_1\rangle|.\]
 Also, $ |\langle x,x_2\rangle|=1+a$ and
 \[ |\langle y,x_2\rangle|=|1+a(\frac{-2}{a}-1)|=|-(1+a)|=1+a.
 \]
 So $x_1,x_2$ fails weak phase retrieval.
 \vskip12pt
 \noindent {\bf Case 2}: After normalizing we have vectors $x_1=(1,0),\ x_2=(1,a)$ with $a<0$.

This time let $x=(1,-1)$ and $y=(1,\frac{-2}{a}+1)$. And a direct computation as above shows that this fails weak phase retrieval.
 \end{proof}

 \begin{lemma}\label{T:X6}
 If $0<ab,cd$ or $ab,cd<0$ then $\{(a,b), (c,d)\}$ fail weak phase retrieval.
 \end{lemma}

 \begin{proof}
 After normalizing, we assume our vectors are $x_1=(1,a),\ x_2=(1,b)$ with $0<a,b$ and $1\le a$. The other case is $\{(a,1),(b,1)\}$ with
 $a,b>0$ and $a>b$ and its proof is symmetric. Now let $x=(a_1,a_2)=(1,1)$ and $y=(b_1,b_2)=
 (1+a-\frac{2a+a^2+ab}{a-b},\frac{2+a+b)}{a-b})$. Then
 $a_1a_2>0$ and since $a>1$ and $b>0$, $b_1b_2<0$. Also,
 \[ \langle x,x_1\rangle = 1+a\mbox{ and }\langle y,x_1\rangle = 1+a,
 \]
 and
 \[ \langle x,x_2\rangle = 1+b\mbox{ and }\langle y,x_2\rangle = -1-b.\]
 therefore $|\langle x,x_i\rangle|=\langle y,x_i\rangle|$ for $i \in [2]$.
 we have that $a_1a_2>0$ and $b_2>0$. So this family fails weak phase retrieval if $b_1<0$. But
 \[ a+b+2ab>0,\]
 and so
 \[2a+2ab>a-b,\]
 and so
 \[ a(2+a+b)=2a+a^2+ab> a-b+a^2-ab=(a-b)(1+a),\]
 and since $a-b>0$,
 \[ \frac{a}{a-b}(2+a+b)>1+a,\]
 and so
\[ b_1=1+a-\frac{a}{a-b}(2+a+b)<0.\]
 \end{proof}

 \begin{lemma}\label{L3}
 If $ab>0$ and $cd<0$ then $\{(a,b), (c,d)\}$ do weak phase retrieval if and only if $a=b,\ c=-d$ or $a=c,\ b=-d$.
 \end{lemma}

 \begin{proof}
\noindent $\Leftarrow$:
 After normalizing, in the first case we may assume our vectors are $(1,1),\ (1,-1)$ and we know that
 these vectors do phase retrieval.
  \vskip12pt
In the second case we have $a=c$ so our two vectors are $x_1=(1,a),\ x_2=(1,-a)$.
 \vskip12pt
 In this case, let $x=(a_1,a_2)$ and $y=(b_1,b_2)$ and assume
 \[ |\langle x,x_1\rangle|=|\langle y,x_1\rangle|\mbox{ and }|\langle x,x_2\rangle|=|\langle y,x_2\rangle|.\]
 There are two cases:
 \vskip12pt
 \noindent {\bf Case 1}: We have
 \[ \langle x,x_1\rangle = \langle y,x_1\rangle\mbox{ and }\langle x,x_2\rangle = \langle y,x_2\rangle.\]
 \vskip12pt
 In this case,
 \[
 \langle x,x_1\rangle = a_1+aa_2
 = \langle y,x_1\rangle
 = b_1+ab_2.\]
And
\[ \langle x,x_2\rangle = a_1-aa_2=\langle y,x_2\rangle = b_1-ab_2.\]
 Adding these two equations together we get:
 \[ 2a_1=2b_1,\mbox{ and so } a_1=b_1\mbox{ and }a_2=b_2.\]
 That is, $a_1a_2=b_1b_2$ have the same signs.
 \vskip12pt
 \noindent {\bf Case 2}: We have
 \[ \langle x,x_1\rangle = \langle y,x_1\rangle\mbox{ and }\langle x,x_2\rangle = -\langle y,x_2\rangle.\]
 \vskip12pt
 In this case we have
 \[ \langle x,x_1\rangle = a_1+aa_2=\langle y,x_1\rangle = b_1+ab_2,\]
 and
 \[ \langle x,x_2\rangle = a_1-aa_2=-\langle y,x_2\rangle = -b_1+ab_2.\]
 Adding these equations together we get:
 \[ 2a_1=2ab_2 \mbox{ and so } b_2=\frac{a_1}{a} ,
 \]
 And subtracting the second equation from the first we get:
 \[ 2aa_2=2b_1\mbox{ and so }b_1=aa_2.\]
 So
 \[ b_1b_2= aa_2\frac{a_1}{a}=a_1a_2.\]

 \vskip12pt
 In both cases, $a_1a_2=b_1b_2$ and so our vectors do weak phase retrieval.
 \vskip12pt
  \noindent $\Rightarrow$:
By our assumptions, after normalizing the vectors we must have $(1,a),\ (1,-b)$ or $(a,1),\ (-b,1)$ with $a\not= b$ We may assume
$a>b$ for  with $a> b>0$. For if not, switch to $(\frac{1}{a},1),\ (\frac{-1}{b},1)$ and $\frac{1}{a}>\frac{1}{b}$ and the argument
in this case is the same as the other case. So let $x_1=(1,a),\ x_2=(1,-b)$ with $a>b$. Also,
  let $x=(a_1,a_2)=(1,a_2)$ and $y=(b_1,b_2)$ where we choose $0<a_2$ small
 enough so that $2aba_2<a-b$. Finally, let
 \[ b_1=1+aa_2-\frac{a}{a+b}[2+aa_2-a_2b]\mbox{ and }b_2=\frac{2+aa_2-a_2b}{a+b}.\]
 We note that since $a>b$ and $a,b,a_2>0$ we have
 \[ b_2=\frac{2+aa_2-a_2b}{a+b}= \frac{2+a_2(a-b)}{a+b}>0.\]
 We now show that $b_1<0$. By our assumption,
 \[ 2aba_2<a-b,\]
 and so
 \[ aba_2<a-b(1+aa_2),\]
 and so
 \[ a+b+a^2a_2+aba_2<2a+a^2a_2-aa_2b\]
 and so
 \[ (1+aa_2)(a+b)< 2a+a^2a_2 -a_2ba\]
 and so
 \[1+aa_2<\frac{a}{a+b}[2+aa_2 -a_2b].\]
 Hence, $b_1<0$. So $a_1a_2>0$ and $b_1b_2<0$.
 Now we compute: $ \langle x,x_1\rangle = 1+aa_2$ while
 \begin{align*}
 \langle y,x_1\rangle &= b_1+ab_2\\
 &= 1+aa_2-\frac{a}{a+b}[2+aa_2-a_2b]+a\frac{2+aa_2-a_2b}{a+b}\\
 &= 1+aa_2=\langle x,x_1\rangle.
 \end{align*}
 Next, $\langle x,x_2\rangle = 1-a_2b$, while
 \begin{align*}
 \langle y,x_2\rangle&= 1+aa_2-\frac{a}{a+b}[2+aa_2-a_2b]-b\frac{2+a_2(a-b)}{a+b}\\
 &= 1+aa_2-\frac{2+a_2(a-b)}{a+b}(a+b)\\
 &= 1+aa_2-2-a_2(a-b)\\
 &= -1+a_2b\\
 &= -\langle x,x_2\rangle.
 \end{align*}
 It follows that our set fails weak phase retrieval.

  \end{proof}
 \section{Weak Phase Retrievable Frames Must Span}
 We now solve a longstanding open problem in the field.

 \begin{theorem}\label{T:X2}
 If $\{x_i\}_{i=1}^m$ is a weak phase retrievable frame in $\RR^n$ then $span\{x_i\}_{i=1}^m=\RR^n$.
 \end{theorem}

 \begin{proof}
 We will prove the contrapositive. I.e. We will assume that $\{x_i\}_{i=1}^m$ does not span and show that it cannot do
 weak phase retrieval.
 We will examine two cases.
 \vskip12pt
 \noindent {\bf Case 1}: There is a vector $z\perp span\{x_i\}_{i=1}^m$ and a vector $x\in span\{x_i\}_{i=1}^m$ so
 that they have two coordinates in common which are not zero.
 \vskip12pt
 Since $z \perp x$,
 \[ \sum_{k=1}^nz(k)x(k)=0.\]
  So there is some $1\le i \not= j \le n$ and $z(i),z(j)\not= 0$ and $x(i),x(j)\not= 0$ and one of the following holds:
  \[ x(i),x(j)>0\mbox{ and }z(i)>0,\ z(j)<0.\]
  Choose $0<a$ so that $az(j)+x(j)<0$. Then $az(i)+x(i)>0$ so
  \[ (az(i)+x(i))(az(j)+x(j))<0\mbox{ while }x(i)x(j)>0.\]
  Also,
  \[ \langle  az+x,x_i\rangle = \langle x,x_i\rangle \mbox{ for all }i=1,2,\ldots.\]
  So weak phase retrieval fails.

  The other possibility is that $x(i)>0,\ x(j)<0$ and $z(i),\ z(j)>0$. Now choose $a>0$ so that $az(j)+x(j)>0$. Now the
  proof follows as above.
  \vskip12pt
  \noindent {\bf Case 2}: If case one fails, then whenever $0\not= x,z$ and $z\perp x\in span\{x_i\}_{i=1}^m$ then $z,x$ do not have
  two non-zero coordinates in common. But since they are orthogonal, they cannot have only one non-zero coordinate
  in common. Hence, their supports are disjoint. So consider the vectors $x+z,\ x-z$. Choose $i,j$ so that $x(i)\not= 0$ and so $z(i)= 0$,
  Then
  \[ \langle x+z,x_i\rangle = \langle x,x_i\rangle = \langle x-z,x_i\rangle \mbox{ for all }i=1,2,\ldots.\]
  But if we choose $i\in support \ x$ and $j\in support \ z$ then
  \[ (x+z)(i) (x+z)(j)=x(i)z(j)\mbox{ and }(x-z)(i)(x-z)(j)=x(i)(-z(j)).\]
  So our frame fails weak phase retrieval.
 \end{proof}

\section{Weak phase retrieval by projections}
Now we present the concept of ``Weak phase retrieval by projections''.
\begin{definition}\label{D: Real weak phase retrieval by projection}
Let $\{W_i\}_{i=1}^m$ be a family of subspaces of $\mathbb{R}^n$ with respective projections $\{P_i\}_{i=1}^m$. This family does weak phase retrieval if given $x,y \in \mathbb{R}^n$ satisfying \\
$\|P_ix\|=\|P_iy\|$, for all $i \in [m]$,  then $x$ and $y$ weakly have the same phase.
\end{definition}
From Definition \ref{D: Real weak phase retrieval by projection}, it is obvious that if a family of subspaces $\{W_i\}_{i=1}^m$  does phase retrieval in $\mathbb{R}^n$, then it does weak phase retrieval too.

We need a result from \cite{ACNT2015}.
\begin{theorem} \label{T:weak is full 1}
  If $X=\{x_i\}_{i=1}^{2n-2}$ does weak phase retrieval in $\mathbb{R}^n$ then $X$ is full spark.
\end{theorem}

The converse to this theorem fails.

 \begin{example}
  Now we  give an example of a full spark not weak phase retrievable frame in  $\mathbb{R}^3$
  with 4 vectors. The frame $\{x_i\}_{i=1}^4$ given by
   $$x_1=(1,0,0),x_2=(0,1,0),x_3=(0,0,1),x_4=(1,1,-3)$$ 
   is full spark but it can not do weak phase retrieval for $\mathbb{R}^3$. Take $x=(4,3,1)$ and $y=(4,-3,-1)$, then $|\langle x,x_i \rangle|=|\langle y,x_i \rangle|$ for $i \in [4]$, but $x$ and $y$ do not weakly have the same sign.
   \end{example}

\begin{proposition}
  If the frame $\{x_i\}_{i=1}^m$ does weak phase retrieval in $\mathbb{R}^n$ then $\{Px_i\}_{i=1}^m$ yields weak phase retrieval for all orthogonal projections $P$ on $\mathbb{R}^n$.
\end{proposition}
\begin{proof}
   Let $x,y \in P(\mathbb{R}^n)$ such that $|\langle x,Px_i \rangle|^2=|\langle y,Px_i \rangle|^2$ for all $i \in [n]$. For
all $i \in [n]$, we have
$$|\langle x,x_i \rangle|^2=|\langle Px,x_i \rangle|^2=|\langle x,Px_i \rangle|^2=|\langle y,Px_i \rangle|^2=|\langle Py,x_i \rangle|^2=|\langle y,x_i \rangle|^2.$$
Since $\{x_i\}_{i=1}^m$ does weak phase retrieval, $x,y$ weakly have the same sign.
\end{proof}

\begin{theorem}\label{T: T12}\cite{Ed2017}
  If $n=2^k-1$, for any $k \in \mathbb{N}$, then it takes $2n-1$ subspaces of $\mathbb{R}^n$ to do phase retrieval.
\end{theorem}
For $n \ne 2^k-1$, for any $k$, we have the following fact:
\begin{theorem} \label{T:T13}\cite{ACCHTTX17}
  It takes at least $2n-2$ hyperplanes in $\mathbb{R}^n$ to do phase retrieval.
\end{theorem}
\begin{theorem}\label{T14}\cite{Ed2017}
  A family of projections $\{P_i\}_{i=1}^m$ does phase retrieval in $\mathbb{R}^n$ if and only if for every $0 \ne x \in \mathbb{R}^n$, $span\{P_ix\}_{i=1}^m=\mathbb{R}^n$.
\end{theorem}
See \cite{PP} for an elementary proof of Theorem \ref{T14}.
\begin{theorem}\label{T12}
If $\{x_i\}_{i=1}^2$ does weak phase retrieval in $\RR^2$ then $\{x_1^{\perp},x_2^{\perp}\}$ do weak
phase retrieval in $\RR^2$. 
\end{theorem}

\begin{proof}
By Theorem \ref{T10} we have two cases: $(1,1),\ (1,-1)$ or $(1,b),\ (1,-b)$. In the first case the perps are also
$(1,1),\ (1,-1)$ and they do weak phase retrieval and in the second case the perps are $(-b,1),\ (b,1)$ which after
normalization are $(1,c),\ (1,-c)$ and this set does weak phase retrieval.
\end{proof}

\begin{example}
,By  \cite{ACGJT2016} the vectors $$\phi_1=(1,1,1),\phi_2=(-1,1,1),\phi_3=(1,-1,1),\phi_4=(1,1,-1)$$ do weak phase retrieval for $\mathbb{R}^3$.

  Let $\{W_i\}_{i=1}^4=\{\phi_i^{\perp}\}_{i=1}^4$ be the set of $4$ hyperplanes so that \\
 $$W_1=\{\phi_1^{\perp}\}=\{(x_1,x_2,x_3) \in \mathbb{R}^3 : x_1+x_2+x_3=0\}$$
  $$W_2=\{\phi_2^{\perp}\}=\{(x_1,x_2,x_3) \in \mathbb{R}^3 : -x_1+x_2+x_3=0\}$$
  $$W_3=\{\phi_3^{\perp}\}=\{(x_1,x_2,x_3) \in \mathbb{R}^3 : x_1-x_2+x_3=0\}$$
  $$W_4=\{\phi_4^{\perp}\}=\{(x_1,x_2,x_3) \in \mathbb{R}^3 : x_1+x_2-x_3=0\}$$
  Let $P_i$ be the orthogonal projection onto $W_i$. then for any $x=(a,b,c)$ we have
  $$P_1(a,b,c)=\left ( a-\frac{a+b+c}{3},b-\frac{a+b+c}{3},c-\frac{a+b+c}{3}\right )$$
  $$P_2(a,b,c)=\left ( -a-\frac{-a+b+c}{3},b-\frac{-a+b+c}{3},c-\frac{-a+b+c}{3}\right )$$
   $$P_3(a,b,c)=\left ( a-\frac{a-b+c}{3},-b-\frac{a-b+c}{3},c-\frac{a-b+c}{3}\right )$$
   $$P_4(a,b,c)=\left ( a-\frac{a+b-c}{3},b-\frac{(a+b-c)}{3},-c-\frac{a+b-c}{3}\right )$$
   It is easy to check that  $span\{P_i(1,1,1)\}_{i=1}^4 \ne \mathbb{R}^3$ and by Theorem \ref{T: T14},  $\{W_i\}_{i=1}^4$ can not do phase retrieval in $\mathbb{R}^3$. But these hyperplanes do weak phase retrieval. The proof of this is straightforward
   but labor intensive so we leave it to the reader.
   
   \end{example}

Even if  $\{x_i\}_{i=1}^{2n-2}$ is not  a full spark frame in $\mathbb{R}^n$,  the hyperplanes $\{W_i\}_{i=1}^{2n-2}=\{x_i^{\perp}\}_{i=1}^{2n-2}$  may still do weak phase retrieval by projections.
\begin{example} Consider the frame $\{x_i\}_{i=1}^4$ with vectors  $$x_1=(1,1,0),x_2=(-1,0,1),x_3=(1,-1,0),x_4=(0,1,-1)$$
  Then the frame $\{x_i\}_{i=1}^4$ is not full spark since $x_4=-x_2-x_3$. Also this frame does not do weak phase retrieval for $\mathbb{R}^3$. Take $x=(-2,-1,0)$ and $y=(1,2,3)$, then $|\langle x,x_i \rangle|=|\langle y,x_i \rangle|$ for $i \in [4]$, but $x$ and $y$ do not weakly  have the same phase.\\
  Let  $\{W_i\}_{i=1}^{4}=\{\phi_i^{\perp}\}_{i=1}^{4}$   so that:
  $$W_1=\{\phi_1^{\perp}\}=\{(x_1,x_2,x_3) \in \mathbb{R}^3 : x_1+x_2=0\}$$
  $$W_2=\{\phi_2^{\perp}\}=\{(x_1,x_2,x_3) \in \mathbb{R}^3 : -x_1+x_3=0\}$$
  $$W_3=\{\phi_3^{\perp}\}=\{(x_1,x_2,x_3) \in \mathbb{R}^3 : x_1-x_2=0\}$$
  $$W_4=\{\phi_4^{\perp}\}=\{(x_1,x_2,x_3) \in \mathbb{R}^3 : x_2-x_3=0\}.$$
  Let $P_i$ be the orthogonal projection onto $W_i$. then for any $x=(a,b,c)$ we define \\
  \[ P_1(a,b,c)=\left ( \frac{a-b}{2},\frac{b-a}{2},c\right )\]
  $$P_2(a,b,c)=\left (\frac{a+c}{2},b, \frac{a+c}{2}\right )$$
   $$P_3(a,b,c)=\left (\frac{a+b}{2},\frac{a+b}{2},c\right )$$
   $$P_4(a,b,c)=\left (a,\frac{b+c}{2},\frac{b+c}{2}\right ).$$
  then these $4$ hyperplanes do weak phase retrieval. Again, we leave the straightforward proof to the reader.
\end{example}

It is known that it takes 2n-1 vectors to do phase retrieval \cite{CCPW13}. Also, phase retrieval can be done with any
family of 2n-1 projections of arbitrary rank between one and n-1 \cite{CCPW13}. There is only one case (n=4) where we 
know we can
do phase retrieval with 2n-2 projections \cite{ACCHTTX17}. We don't know the minimal number of projections it
takes to do phase retrieval or weak phase retrieval.
It is also known that when $n=2^k-1$ then it takes 2n-1 subspaces
to do phase retrieval \cite{Ed2017}. All other cases are open.

\section{Existence of Weak Phase Retrievable Frames And Projections}

Weak phase retrieval does not pass to the orthogonal complement hyperplanes.

\begin{example}
In \cite{ACCHTTX17}There is an example in $\RR^3$: 
\[ x_1=(0,0,1),\ x_2=(1,0,1),\ x_3=(0,1,1),\ x_4=(1,1-\sqrt{2},2),\ x_5=(1,1,1).\]
Since these 5 vectors are full spark, $\{x_i\}_{i=1}^5$ does phase retrieval in $\RR^3$.
They observe that for the projections $\{P_i\}_{i=1}^5$ onto the orthogonal complement hyperplanes,
$span\{P_i(x_5)\}_{i=1}^{5}=span[e_1,e_2]$. It follows that
\[ \langle P_i(x_5),2e_3\rangle = \langle P_i(x_5),P_i(2e_3)\rangle =\langle P_i(x_5),-P_i(2e_3)\rangle=0,\mbox{ for all }i=1,2,\ldots,5.\]
Hence,
\[ \|P_i(x_5+2e_3)\|^2=\|P_i(x_5)\|^2+\|P_i(2e_3)\|^2=\|P_i(x_5-2e_3)\|^2.
\]
But $x_5+2e_3=(1,1,3)$ and $x_5-2e_3)=(1,1,-1)$. Since these vectors don't have the same or opposite signs, $\{P_i\}_{i=1}^5$
fails weak phase retrieval.
\end{example}

What is needed here is norm retrieval.
Our first theorem is known for phase retrieval \cite{CCPW13}.

\begin{theorem}
If $\{P_i\}_{i=1}^m$ does weak phase retrieval and $\{(I-P)x_i\}_{i=1}^m$ does norm retrieval, then $\{(I-P_i)\}_{i=1}^m$
does weak phase retrieval.
\end{theorem}

\begin{proof}
Assume $x,y\in \RR^n$ and $\|(I-P)_ix\|=\|(I-P_i)y\|$ for all $i=1,2,\ldots,m$. By assumption, $\|x\|=\|y\|$. Hence,
\[ \|x\|^2=\|P_ix\|^2+\|(I-P_i)x\|^2=\|y\|^2=\|P_iy\|^2+\|(I-P_i)y\|^2.\]
It follows that $\|P_ix\|=\|P_iy\|$ for all $i=1,2,\ldots,m$. Since $\{P_i\}_{i=1}^m$ does weak phase retrieval, $x,y$
have the same phase. So $\{(I-P_i)\}_{i=1}^m$ does weak phase retrieval.
\end{proof}

If a frame does phase retrieval, it does norm retrieval.
But, in general, a frame can do weak phase retrieval but fail norm retrieval. Also, it is possible that $\{P_i\}_{i=1}^m$
and $\{(I-P_i)\}_{i=1}^m$ both do weak phase retrieval but neither does norm retrieval. 

\begin{example}
In $\RR^2$ let $x_1=(1,b),\ x_2=(1,-b)$ with $b\not= 1$. By Theorem \ref{T10} this family does weak phase retrieval.
Let $x=(1,1),\ y=(b,\frac{1}{b})$. Then
\[ \langle x,x_1\rangle = 1+b\mbox{ and } \langle y,x_1\rangle = b+b\frac{1}{b}=1+b.\]
Also,
\[ |\langle x,x_2\rangle|=1-b\mbox{ and }|\langle y,x_2\rangle|= |b-b\frac{1}{b}|=|b-1|.\]
Since $b\not= 1$, $\|x\|\not= \|y\|$.

Note also that if $P_i$ projects on $x_i$ then $\{P_i\}_{i=1}^2$ does weak phase retrieval and $(I-P_i)$ projects
on $x_i^{\perp}$ and so also does weak phase retrieval by Theorem \ref{T12}.
\end{example}

For the next theorem we need a result from \cite{ACNT2015}.

\begin{theorem}\label{T51}
Let $x=(a_1,a_2,\ldots,a_n),\ y=(b_1,b_2,\ldots,b_n)\in \RR^n$. If there exists an $i\in [n]$ so that $a_ib_i\not= 0$ and
$\langle x,y\rangle =0$, then $x,y$ do not weakly have the same signs or opposite signs.
\end{theorem}

\begin{theorem}\label{T52}
If $\{x_i\}_{i=1}^{2n-2}$ does weak phase retrieval in $\RR^n$ then for every $I\subset [2n-2]$ with $|I|=n-1$,
if $x\perp span\{x_i\}_{i\in I}$ and $y\perp \{x_i\}_{i\in I^c}$ then $\frac{x}{\|x\|}+\frac{y}{\|y\|}$ and $\frac{x}{\|x\|}-\frac{y}{\|y\|}$ are disjointly
supported. In particular, if $\|x\|=\|y\|=1$ then $x+y$ and $x-y$ are disjointly supported.
\end{theorem}

\begin{proof}
We have
\[ \langle \frac{x}{\|x\|}-\frac{y}{\|y\|},\frac{x}{\|x\|}+\frac{y}{\|y\|}\rangle = \|\frac{x}{\|x\|}\|^2-
\langle \frac{y}{\|y\|},\frac{x}{\|x\|}\rangle +\langle \frac{x}{\|x\|},\frac{y}{\|y\|}\rangle -\|\frac{y}{\|y\|}\|^2=0.\]
For $i\in I$,
\[ |\langle \frac{x}{\|x\|}+\frac{y}{\|y\|},x_i\rangle|=|\langle \frac{y}{\|y\|},x_i\rangle|=|\langle \frac{-y}{\|y\|},x_i\rangle| =
|\langle \frac{x}{\|x\|}-\frac{y}{\|y\|},x_i\rangle|\]
For $i\in I^c$,
\[ |\langle \frac{x}{\|x\|}+\frac{y}{\|y\|},x_i\rangle|=|\langle \frac{x}{\|x\|},x_i\rangle|=
|\langle \frac{x}{\|x\|}-\frac{y}{\|y\|},x_i\rangle|\]

Since $\{x_i\}_{i=1}^{2n-2}$ does weak phase retrieval, by Theorem \ref{T51}, $\frac{x}{\|x\|}+\frac{y}{\|y\|}$ and
$\frac{x}{\|x\|}-\frac{y}{\|y\|}$ are disjointly supported.
\end{proof}

The following theorem shows how close weak phase retrieval is to phase retrieval.
For notation, if $x\in\RR^n$ and $I\subset [1,n]$ then $x_{I}=\sum_{i\in I}x(i)e_i$.

\begin{theorem}
Let $\{x_i\}_{i=1}^{2n-2}$ be full spark vectors in $\RR^n$. The following are equivalent:
\begin{enumerate}
\item $\{x_i\}_{i=1}^{2n-2}$ does weak phase retrieval.
\item If $|\langle x,x_i\rangle|=|\langle y,x_i\rangle|$ for all $i=1,2,\ldots,2n-2$ then there is an $0 \ne a\in \RR$
so that for every $i=1,2,\ldots,2n-1$ either $x(i)=ay(i)$ or $x(i)=\frac{1}{a}y(i)$. I.e. $x=ay_{I}+\frac{1}{a}y_{I^c}$
for some $I\subset [1,n]$.
\end{enumerate}
\end{theorem}

\begin{proof}
$(1)\Rightarrow (2)$: If $|\langle x,x_i\rangle|=|\langle y,x_i\rangle|$ for all $i=1,2,\ldots,2n-2$
if $x=y$ or $x=-y$ (2) is immediate. Otherwise, let
\[ I = \{1\le i \le 2n-2:\langle x,x_i\rangle = \langle y,x_i\rangle\}.\]
So $x-y \perp span\{x_i\}_{i\in I}$. Similarly, $x+y\perp span\{x_i\}_{i\in I^c}.$
If $span\{x_i\}_{i\in I}=\RR^n$ then $x=y$ and if $span\{x_i\}_{i\in I^c}=\RR^n$ then $x=-y$ so (2) holds.
Otherwise, since the frame is full spark, it follows that $|I|=n-1$. By Theorem \ref{T52}, we have that
\[ \frac{x+y}{\|x+y\|}+\frac{x-y}{\|x-y\|}\mbox{ and } \frac{x+y}{\|x+y\|}-\frac{x-y}{\|x-y\|}\mbox{ are disjointly supported}.\]
So for every $i=1,2,\ldots,2n-2$,
\[ \mbox{either }(\frac{x+y}{\|x+y\|}+\frac{x-y}{\|x-y\|})(i)=0\mbox{ or } (\frac{x+y}{\|x+y\|}-\frac{x-y}{\|x-y\|})(i)=0.\]
If
\[ (\frac{x+y}{\|x+y\|}-\frac{x-y}{\|x-y\|})(i)=0
\]
then
\[ \left ( \frac{1}{\|x+y\|}+\frac{1}{\|x-y\|}\right )x(i)=\left ( \frac{-1}{\|x+y\|}+\frac{1}{\|x-y\|}\right )y(i)
\]
Let
\[ a=\frac{\left ( \frac{-1}{\|x+y\|}+\frac{1}{\|x-y\|}\right )}{\left ( \frac{1}{\|x+y\|}+\frac{1}{\|x-y\|}\right )}.\]
Then $x(i)=ay(i)$. In the second case:
\[ (\frac{x+y}{\|x+y\|}+\frac{x-y}{\|x-y\|})(i)=0,\]
and so
\[ \left ( \frac{1}{\|x+y\|}-\frac{1}{\|x-y\|}\right )x(i)=-\left ( \frac{1}{\|x+y\|}+\frac{1}{\|x-y\|}\right )y(i)
\]
So
\[ x(i) = \frac{-\left ( \frac{1}{\|x+y\|}+\frac{1}{\|x-y\|}\right )}{\left ( \frac{1}{\|x+y\|}-\frac{1}{\|x-y\|}\right )}y(i)=\frac{1}{a}y(i).\]
This proves (2)
\vskip12pt

$(2)\Rightarrow (1)$: If $|\langle x,x_i\rangle|=|\langle y,x_i\rangle|$ for all $i=1,2,\ldots,2n-2$ then by (2) there is
an $a\in \RR$ so that for every i, either $x(i)=ay(i)$ or $x(i)=\frac{1}{a}y(i)$. Hence, if $a>0$ then $x,y$ have the
same signs and if $a<0$ then $x,y$ have opposite signs.
\end{proof}

So the difference between weak phase retrieval and phase retrieval is that for phase retrieval we must have
that $a=\pm 1$ in the theorem.

\section{phase (norm) retrieval by projections}
Now we present some facts about phase (norm) retrieval by projections.
\begin{theorem} \label{T: t1}
In $\mathbb{R}^n$  two subspaces can do norm retrieval.
\end{theorem}
\begin{proof}
Choose a subspace $W$ and $W^{\perp}$.
\end{proof}

We know phase retrievability by vectors is preserved by invertible operators, but phase retrievability by higher dimensional subspaces is not
   preserved by invertible operators in general  by  \cite{CCJLW15}. Also  We know norm retrievability by vectors is not preserved by invertible operators in general \cite{CGJT17}. The case for subspaces is much more complicated.
\begin{theorem}
If $\{W_1,W_2\}$ is a fusion frame for $\RR^n$ then one of the following holds:
\begin{enumerate}
\item If $W_1\cap W_2=\{0\}$, then there is an invertible operator $T$ on $\RR^n$ so that $\{TW_1,TW_2\}$ does norm retrieval.
\item If $W_1\cap W_2=W_3\not= \{0\}$ then there is no invertible operator $T$ on $\RR^n$ so that $\{TW_1,TW_2\}$ does norm retrieval.
\end{enumerate}
\end{theorem}

\begin{proof}
 Let $\{e_i\}_{i=1}^n$ be an orthonormal basis for $\RR^n$.
 \vskip12pt
 (1)
In this case, there exists an invertible operator $T$ on $\RR^n$ so that $TW_1=span\{e_i\}_{i=1}^k$ and $TW_2=span\{e_i\}_{i=k+1}^n$.
Clearly $\{TW_i\}_{i=1}^2$ does norm retrieval.
\vskip12pt
(2) Let $T$ be an invertible operator on $\RR^n$. Let $W_1'$ be the orthogonal complement of $TW_3$ in  $TW_1$ and
$W_2'$ the orthogonal complement of $TW_3$ in $TW_2$.
\vskip12pt
We examine two cases.
\vskip12pt

\noindent {\bf Case 1}: $W_1'\perp W_2'$.

Then there is an orthogonal set of vectors $x_1,x_2,x_3$ with
$\|x_i\|=1$ for $i=1,2,3$, $x_1\in W_1'$,  $x_2\in TW_3$, and $x_3\in W_2'$. Define
\[ y_1=x_1+2x_2+x_3 \mbox{ and } y_2=2x_1+x_2+2x_3.\]
Let $\{P_i\}_{i=1}^2$ project onto $\{TW_i\}_{i=1}^2$.
Then
\[ \|P_1y_1\|^2=1^2+2^2=2^2+1^2=\|P_1y_2\|^2.
\]
Also,
\[ \|P_2y_2\|^2= 2^2+1^2=1^2+2^2=\|P_2y_2\|^2.\]
But
\[ \|y_1\|^2= 1^2+2^2+1^2=6\mbox{ and }\|y_2^2\|^2=2^2+1^2+2^2=9.\]
So norm retrieval fails.
\vskip12pt
\noindent {\bf Case 2:} $W_1'$ is not orthogonal to $W_2'$.

Then there is a vector $x\in W_1'$ with $\|x\|=1$ and $0\not= P_2x\in W_2'$.
Now
\[ \{y:P_1y=P_1x=x\}=\{y=x+z: z\perp x\},\]
and
\[ \{y:P_2y=P_2x\}=\{y=P_2x+z:z\perp P_2x\}.\]
Since our assumptions imply $n\ge 3$, we have a $\|z\|=1$ with $z\in x^{\perp}\cap (P_2x)^{\perp}$. Now,
$x=P_1x=P_1(x+z)$ and $P_2x=P_2(x+z)$. Hence
\[ \|x\|=\|P_1x\|=\|P_2(x+z)\|\mbox{ and } \|P_2x\|=\|P_2(x+z)\|.\]
But
\[ \|x+z\|^2=\|x\|^2+\|z\|^2 \not= \|x\|^2\mbox{ since }z\not= 0.\]
So norm retrieval fails.
\end{proof}
\begin{theorem} \label{T:C1}
  \cite{CCJLW15} Let $\{P_i\}_{i=1}^m$ be projections onto  subspaces $\{W_i\}_{i=1}^m$ of $\mathbb{R}^n$. Then the following are equivalent:
  \begin{enumerate}
\item $\{W_i\}_{i=1}^m$ does norm retrieval.
\item For any orthonormal basis $\{\phi_{ij}\}_{j=1}^{I_i}$  of $W_i$, then the collection of vectors $\{\phi_{ij}\}_{j=1,i=1}^{I_i,m}$
do norm retrieval.
\end{enumerate}
\end{theorem}
In light of above Theorem, it is natural to ask if we change the orthonormal basis condition with a Riez basis that need not have orthogonal
elements? The answer is negative. we construct a counterexample similar to phase retrieval case \cite{CCJLW15}.
\begin{example}
  Let $\{e_i\}_{i=1}^3$ be an orthonormal basis for $\mathbb{R}^3$. Define the subspaces
  $$W_1=[e_1,e_2], \: W_2=[e_2], \: W_3=[e_3]$$
  $$W_4=[\frac{e_1+e_2}{2}], \: W_5=[\frac{e_2+e_3}{2}], \: W_6=[\frac{e_1+e_3}{2}]$$
  Then $\{W_i\}_{i=1}^6$ can do phase (norm) retrieval in $\mathbb{R}^3$.
  However if we choose not orthonormal Riesz basis $\{e_1+e_2, e_2\}$ for $W_1$ and the spanning element from the other subspaces, we get the set of vectors $\{e_1+e_2, e_2,e_3, e_2+e_3,e_1+e_3\}$. Notice that this set cannot do norm retrieval in $\mathbb{R}^3$, since if we consider two partitions $\{e_2,e_3, e_2+e_3\}$ and $\{e_2+e_3,e_1+e_3\}$ then $e_1 \in \{e_2,e_3,e_2+e_3\}^{\perp}$ and $e_1-e_2-e_3 \in \{e_2+e_3,e_1+e_3\}^{\perp}$ but $\langle e_1,e_1-e_2-e_3 \rangle \ne 0$
\end{example}

  Similar to the vector case,  norm retrievability by higher dimensional subspaces is not preserved by invertible operators.
    \begin{example}
   We know every subspace plus its perpendicular subspace do norm retrieval in $\mathbb{R}^n$. Then the subspaces
   $\{W_i\}_{i=1}^2=\{[e_1],[e_2,e_3]\}$ do norm retrieval for $\mathbb{R}^3$. Define the invertible operator T on the basis elements by:
  $$ Te_i=
\begin{cases}
  e_1-e_2   & \mbox{if } i=1 \\
  e_2      & \mbox{if } i=2 \\
  e_3      & \mbox{if} \: i=3
\end{cases} $$
Then we have
$TW_1=[e_1-e_2], \: TW_2=[e_2,e_3]$ and $\{TW_1\}_{i=1}^2$ do not norm retrieval in $\mathbb{R}^3$ since by \cite{CGJT17} two 2- dimensional subspaces that are not perpendicular, cannot do norm retrieval in  $\mathbb{R}^3$.
\end{example}

\section{Open Problems}

It is known that the phase retrievable frames are not dense in all frames.

\begin{problem}
Show that the weak phase retrievable frames with 2n-2 vectors in $\RR^n$ are not dense in all frames with 2n-2 vectors.
\end{problem}

We have examples of weak phase retrievable frames 
 with 2n-2 vectors for n=3,4 \cite{ACGJT2016}.

\begin{problem}
Show for every $n\ge 5$ that there are weak phase retrievable frames $\{x_i\}_{i=1}^{2n-2}$ for $\RR^n$.
\end{problem}

We know that Phase retrieval in $\RR^n$ is possible using $2n-1$ subspaces each of any dimension less than $n-1$
\cite{CCPW13}.

\begin{problem}
Is weak phase retrieval in $\RR^n$ possible using $2n-2$ (or even $2n-3$) subspaces each of any dimension less than $n$.
\end{problem}

\end{document}